\documentclass[11pt]{amsart} 
\usepackage{amsmath,pb-diagram,amsthm, amscd, amssymb, amsfonts}
\usepackage{enumerate}
\usepackage[dvips]{graphicx}
\usepackage{hyperref}
\usepackage{pictexwd,dcpic}
\newtheorem{teor}{Theorem}[section] \newtheorem{corol}[teor]{Corollary} \newtheorem{prop}[teor]{Proposition} \newtheorem{lem}[teor]{Lemma}

\theoremstyle{definition} \newtheorem{defin}[teor]{Definition}  \newtheorem{ejem}[teor]{Example}

\theoremstyle{remark} \newtheorem{obs}[teor]{Remark}
\newtheorem{rmks}[teor]{Remarks}

\newcommand\st{\operatorname{st}}  \newcommand{\End}{\mbox{\rm End\,}}
\newcommand{\id}{\mathrm{id}}
\newcommand{\ind}{\text{Ind}}
\newcommand{\unit}{\text{\textbf{1}}}
\newcommand\Vect{\operatorname{Vec}}
 
 \newcommand{\Hom}{\text{Hom}} \newcommand{\C}{\mathcal{C}}
  \newcommand{\F}{\mathcal{F}}
\newcommand{\D}{\mathcal{D}}\newcommand{\A}{\mathcal{A}} \newcommand{\M}{\mathcal{M}} \newcommand{\Aut}{{\rm Aut}}
\newcommand{\N}{\mathcal{N}}

\begin{document}

\title[Clifford theory for graded fusion categories]{Clifford theory for graded fusion categories}
\author{C\'esar Galindo}
\address{Departamento de Matem\'aticas\\ Universidad de los Andes\smallbreak Carrera 1 N. 18A-10,  Bogot\'a, Colombia}
\email{cesarneyit@gmail.com, cn.galindo1116@uniandes.edu.co}
\date{}

\begin{abstract}
We develop a categorical analogue of Clifford theory for strongly
graded rings over graded fusion categories. We describe module
categories over a fusion category graded by a group $G$ as induced
from module categories over fusion subcategories associated with the
subgroups of $G$. We define invariant $\C_e$-module categories and
extensions of $\C_e$-module categories. The construction of module
categories over $\C$ is reduced to determining invariant module
categories for subgroups  of $G$ and the indecomposable extensions
of this modules categories. We associate a $G$-crossed product
fusion category to each $G$-invariant $\C_e$-module category and
give a criterion for a graded fusion category to be a
group-theoretical fusion category. We give necessary and sufficient
conditions for an indecomposable module category to be extendable.
\end{abstract}

\maketitle
\section{Introduction and main results}

{\bf 1.} Fusion categories arise in several areas of mathematics and physics, \textit{e.g.}, conformal field theory, operator algebras, representation theory of quantum groups, topological quantum computation, topological quantum field theory, low dimensional topology, and others.
\smallbreak
It is an important and interesting question to classify indecomposable module categories over a given fusion category. The existence of certain module categories yields valuable information about the fusion category $\C$. For example, $\C$ admits a module category of rank one if and only if $\C$ is the category of representations of a semisimple Hopf algebra.

\smallbreak Let $\C$ be a fusion category and let $G$ be a finite
group. We say that $\C$ is graded by $G$ or $G$-graded if $\C
=\oplus_{\sigma\in G} \C_\sigma$, and for any $\sigma, \tau \in G$,
one has $\otimes: \C_\sigma \times \C_\tau \to \C_{\sigma\tau}$.
Graded fusion categories are very important in the study and
classification of fusion categories, see \cite{DGNO,ENO,ENO2,ENO3}.
\smallbreak {\bf 2.} In \cite{G} the author gives the first steps
toward the understanding of  module categories over graded tensor
categories, developing a Clifford theory for a special kind of
graded tensor categories. \smallbreak The main goal of this paper is
to generalize the main results of \cite{G} to arbitrary graded
fusion categories and the description using group-theoretical data
of the indecomposable module categories of a graded fusion category.
\smallbreak Let $\C$ be a  $G$-graded fusion category.  Given a
$\C$-module category  $\M$, we shall denote by $\Omega_{\C_e}(\M)$
the set of equivalence classes of indecomposable $\C_e$-submodule
categories of $\M.$ By Corollary \ref{G-conjuto}, the group   $G$
acts on $\Omega_{\C_e}(\M)$ by $$G\times
\Omega_{\C_e}(\M)\to\Omega_{\C_e}(\M), \  \  (\sigma,[X])\mapsto
[\C_\sigma\boxtimes_{\C_e} X].$$ \smallbreak {\bf 3.} Our first main
result is the Clifford theorem for module categories over fusion
categories:

\begin{teor}\label{theorem clifford}
Let $\C$ be a  $G$-graded fusion category and let $\M$ be an indecomposable $\C$-module category. Then:

\begin{enumerate}
  \item The action of $G$ on $\Omega_{\C_e}(\M)$ is transitive,
  \item Let $\N$ be an indecomposable $\C_e$-submodule subcategory of $\M$. Let $H=st([\N])$ be the stabilizer subgroup of $[\N] \in
 \Omega_{\C_e} (\M)$, and let also \[ \M_\N=\sum_{h\in H}\C_H\overline{\otimes}\N.\] Then $\M_\N$ is an indecomposable $\C_H$-module category and $\M$ is equivalent to $\ind_{\C_H}^\C(\M_\N)$ as $\C$-module categories.
\end{enumerate}
\end{teor}

Here $\C_H$ is the fusion subcategory $\oplus_{\sigma\in
H}\C_\sigma\subset \C$, and $\overline{\otimes}$ is defined in
Section \ref{section clifford theory}.

Applying Theorem \ref{theorem clifford} to the $\C_H$-module
category $\M_\N$ we prove, in Corollary \ref{corol main result},
that every indecomposable module category over a $G$-graded fusion
category $\C$ is equivalent to $\C\boxtimes_{\C_S}\N$, where $\N$ is
an indecomposable $\C_S$-module category  that remains
indecomposable as a $\C_e$-module category, and $S\subset H$ is a
subgroup. Also, in Proposition \ref{prop equivalent inducidas} we
provide a necessary and sufficient conditions for induced module
categories to be equivalent. \smallbreak

{\bf 4.} Let $\C$ be a fusion category. An indecomposable
$\C$-module category is called \emph{pointed module category} if
every simple object in $\C_\M^*$ is multiplicatively invertible (see
subsection \ref{subsection prelim module cat} for the definition of
$\C_\M^*$). A fusion category is called \emph{group-theoretical} if
it admits a pointed $\C$-module category, see \cite{ENO}.
Group-theoretical categories can be explicitly described in terms of
finite groups and their cohomology, see \cite{O2}. \smallbreak Let
$\C=\bigoplus_{\sigma \in G}\C_\sigma$ be a $G$-graded fusion
category. A $\C_e$-module category $\M$, is called $G$-invariant if
$\C_\sigma\boxtimes_{\C_e}\M$ is equivalent to $\M$ as $\C_e$-module
categories, for all $\sigma\in G$. \smallbreak

Our second main result is a criterion for a graded fusion category
to be group theoretical, this generalizes \cite[Theorem
3.5]{Non-grouptheo}.

\begin{teor}\label{teor group-theo}
A $G$-graded fusion category $\C$ is group-theoretical if and only if  $\C_e$ has a $G$-invariant pointed category.
\end{teor}

Consequently, if a Tambara-Yamagami category (see\cite{TY}), or one
of their generalizations in \cite{2Gen-TY} and \cite{Gen-TY}, is the
category of representations of a Hopf algebra, then it is a
group-theoretical fusion category (this result is well-known for TY
categories, see \cite[Remark 8.48]{ENO}). \smallbreak {\bf 5.}
Corollary \ref{corol main result} reduces the construction of
indecomposable $\C$-module categories over a graded fusion category
$\C=\bigoplus_{\sigma\in G}\C_\sigma$, to the construction  of
$\C_H$-module categories $\M$ such that the restriction to $\C_e$
remains indecomposable, for some subgroup $H\subset G$. \smallbreak
If $(\M,\otimes)$ is a $\C_e$-module category, then an extension of
$\M$ is a $\C$-module category $(\M,\odot)$ such that $(\M,\otimes)$
is obtained by restriction to $\C_e$.

\smallbreak Our third main result provides a necessary and
sufficient condition for an indecomposable $\C_e$-module category to
have an extension.

\smallbreak Let $\M$ be an indecomposable $\C_e$-module category and
$\overline{\M}=\ind_{\C_e}^\C(\M)$. Then, by the results of Section
\ref{Section invariant y G graduado}, the fusion category
$\C^*_{\overline{\M}}$ has a natural $G^{op}$-grading.

\begin{teor}\label{teor extensiones}
Let $\C$ be a $G$-graded fusion category. Then an indecompo\-sable
left $\C_e$-module category $\M$ has an extension $(\M,\odot)$ if
and only if $\C^*_{\overline{\M}}$ is a semi-direct product fusion
category. There is a one-to-one correspondence between equivalence
classes of $\C$-extensions of $\M$ and conjugacy classes of graded
tensor functors  Vec$_G\to \C^*_{\overline{\M}}$.

\end{teor}

Consequently, (see Corollary \ref{corol equiv}), if $\M$ is an
extension of an indecomposable $\C_e$-module category, $\C_\M^*$ is
a $G^{op}$-equivariantization of $\C_e$. Finally, in Proposition
\ref{propo decripcion extension en group-theo data} we describe
extensions using group-theoretical data.

\smallbreak {\bf 6} At the same time this paper was completed, Meir
and Musicantov posted the paper \cite{MM} containing results similar
to some of ours. In this paper module categories over $\C$ are
classified in terms of module categories over $\C_e$ and the
extension data $(c, M, \alpha)$ of $\C$ defined in \cite{ENO3}.

\smallbreak {\bf 7} The organization of the paper is as follows: In
Section 2 we discuss preliminaries. In Section 3 we study module
categories graded over a $G$-set and provide a structure theorem for
them. In Section 4   we prove Theorem \ref{theorem clifford}. In
Section 5 we study $G$-graded module categories and invariant module
categories, and we prove Theorem \ref{teor group-theo}. In Section 6
we prove Theorem \ref{teor extensiones}.

\bigbreak \textbf{Acknowledgements} This work was partially
supported by the Project 003568 of the Department of Mathematics,
Pontificia Universidad Javeriana. The author thanks M. Mombelli, J.
Ochoa and E. Rowell for useful discussions and advise.
\section{Preliminaries}

Throughout the paper we work over an algebraically closed field $k$ of characteristic
0.  All categories considered in this paper are finite, abelian, semisimple,
and $k$-linear. All functors and bifunctors considered in this paper are additive and $k$-linear.
\subsection{Fusion categories}

By a fusion category we mean a $k$-linear semisimple rigid tensor category $\C$ with finitely many isomorphism classes of simple objects, finite dimensional spaces of morphisms and such that the unit object of $\C$ is simple. We refer the reader to \cite{ENO} for a general theory of fusion categories.

\begin{ejem}[Examples of fusion categories]
{\bf 1.} The category Vec$_G$ of finite dimensional vector spaces
graded by a finite group $G$. Simple objects in this category are
$\{k_\sigma\}_{\sigma\in G}$, the vector spaces of dimension one
graded by  $\sigma\in  G$, and the tensor product is given by
$k_\sigma \otimes k_\tau = k_{\sigma\tau}$, with the associativity
morphism being the identity.

More generally, choose $\omega\in  Z^3(G, k^*)$ a normalized 3-cocycle. To this 3-cocycle we can attach a
twisted version Vec$_G^\omega$ of Vec$_G$: the simple objects and the tensor product functor
are the same, but the associativity isomorphism is given by $\alpha_{V_\sigma,V_\tau,V_\rho} = \omega(\sigma, \tau, \rho)\id.$
The pentagon axiom then follows from the cocycle condition
\[\omega(\tau, \rho, \nu)\omega(\sigma, \tau \rho, \nu)\omega(\sigma, \tau, \rho) = \omega(\sigma\tau, \rho, \nu)\omega(\sigma, \tau, \rho \nu).\]
Note that cohomologous cocycles define equivalent fusion categories.

{\bf 2.} The category Rep$(H)$ of finite dimensional representations
of a finite dimensional semisimple quasi-Hopf algebra $H$.

{\bf 3.} The category of integrable modules (from category $\mathcal{O}$) over the affine algebra
$\widehat{sl}_2$ at level $l$ (see \cite{BaKi}).
\end{ejem}

By a fusion subcategory of a fusion category $\C$ we understand a full tensor subcategory of $\C$. For any fusion category $\C$, the unit object $\unit$ generates a trivial fusion subcategory equivalent to Vec, the fusion category of finite dimensional vector spaces over $k$.
\subsection{Multiplicatively invertible objects and pointed fusion categories}\label{pointed y obstruccion}

An object $X$ in $\C$ is said to be \textit{multiplicatively
invertible} if $X$ is rigid with a dual object $X^*$  such that
$X\otimes X^* = \unit$. An invertible object is necessarily simple
and the set of isomorphism classes of invertible objects forms a
group: the multiplication is given by tensor products and the
inverse operation by taking dual objects, we shall denote this group
by $U(\C)$.

A fusion category is called \textit{pointed} if all its simple
objects are multiplicatively invertible. Examples of pointed fusion
category are  the fusion categories $\Vect_G^{\omega}$.

Let $\C$ be a fusion category. In order to see whether a pointed
fusion subcategory $\D\subset \C$ is  tensor equivalent to Vec$_G$,
with $U(\D)=G$, we choose a set $\{X_\sigma\}_{\sigma\in G}$ of
representative objects and a family of isomorphisms
$\{t_{\sigma,\tau} : X_\sigma\otimes X_\tau\to
X_{\sigma\tau}\}_{\sigma,\tau \in G}$. Recall, for all simple object
$X\in \C$, $\End_{\C}(X)=k$, thus there exists a unique function
$\omega: G\times G\times G\to k^*$, such that
\[t_{X_{\sigma\tau},X_\rho}\circ(t_{X_\sigma,X_\tau}\otimes\id_{X_\rho})
= \omega(g,h,k)t_{X_\sigma,X_{\tau\rho}}\circ (\id_{X_\sigma}\otimes
t_{X_\tau,X_\rho})\] for all $\sigma, \tau, \rho \in G$. The
function $\omega$ is a 3-cocycle, and the ambiguity of the choice of
$t_{\sigma,\tau}$ gives rise to a coboundary of $G$, the cohomology
class $\omega(\D)\in  H^3(G;k^*)$ is well defined, which is referred
to as the  obstruction of $\D$.

The following proposition is well known and follows from the
previous discussion, we include it for reader's convenience.
\begin{prop}\label{prop obstruccion punteadas}
Let $\D$ be a pointed fusion category. Then $\D$ is tensor
equivalent to Vec$_G^\omega$, where $G=U(\D)$ and $\omega$ is a
3-cocycle in the class of $\omega(\D)$. The pointed fusion category
$\D$ is tensor equivalent to Vec$_G$ if and only if $\omega(\D)=0$.
\end{prop}
\qed

\subsection{Module categories over fusion categories}\label{subsection prelim module cat}

Let $(\C, \, \otimes, \, \text{\textbf{1}}, \, \alpha)$ be a
fusion category, where $\text{\textbf{1}}$ is
the unit object and $\alpha$ is the associativity constraint. Without loss of generality we may assume that $\alpha_{X,\unit,Y}=\alpha_{X,Y,\unit}=\alpha_{\unit,X,Y}=\id_{X\otimes Y}$ for all $X,Y \in C$.

A left {\em $\C$-module category}  (see \cite{O1}) is
a category $\M$ together with a bifunctor $\otimes: \C\times\M \to
\M$ and natural isomorphisms
$m_{X, \, Y \, M}: (X \otimes Y)\otimes M  \to X \otimes (Y \otimes M),$
 for all $M \in \M, \, X, Y \in \C$, such that
the following two equations hold for all $M \in \M, \, X, Y, Z \in \C$:

\begin{align}
    (\alpha_{X,Y,Z}\otimes M)m_{X,Y\otimes Z,M}(X\otimes m_{Y,Z,M}) &= m_{X\otimes Y,
    Z,M}m_{X,Y,Z\otimes M},\label{pentagono module cat}\\
     \unit\otimes M &= M.
\end{align}

For two left $\C$-modules categories  $\M$ and $\N$, a  $\C$-module
functor  $(F,\phi):\M\to \N$ consists of a functor  $F:\M\to \N$
and natural isomorphisms \[\phi_{X,M}:F(X\otimes M)\to X\otimes
F(M),\] such that
\begin{equation}\label{penta funtor modulo}
(X\otimes \phi_{Y,M})\phi_{X,Y\otimes M}F(m_{X,Y,M})=
m_{X,Y,F(M)}\phi_{X\otimes Y, M}
\end{equation}
for all $X, Y\in \C$, $M\in M$.

A $\C$-linear natural transformation  between  $\C$-module functors $(F,\phi),$ $(F',\phi'):\M\to \N$, is
a $k$-linear natural transformation  $\sigma:F\to F'$ such that \[\phi'_{X,M}\sigma_{X\otimes
M}=(X\otimes \sigma_M)\phi_{X,M},\] for all $X\in \C, M\in \M$.

We shall denote the category of $\C$-module functors and $\C$-linear natural transformations between $\C$-modules categories $\M,  \N$ by $\F_\C(\M,\N)$.

Two $\C$-module categories $\M_1$ and $\M_2$ are {\em equivalent}
if there exists a module functor from $\M_1$ to $\M_2$ which is an
equivalence of categories.
\medbreak

For two $\C$-module categories $\M_1$ and $\M_2$  their {\em direct
sum} is the category $\M_1 \oplus \M_2$ with the obvious module
category structure. A module category is {\em indecomposable} if it
is not equivalent to a direct sum of two non-trivial module
categories. It was shown in \cite{O1} that $\C$-module categories
are completely reducible, \textit{i.e.}, given a $\C$-module
subcategory $\N$ of a $\C$-module category $\M$ there is a unique
$\C$-module subcategory $\N'$ of $\M$ such that $\M= \N \oplus \N'$.
Consequently, any $\C$-module category $\M$ has a unique, up to a
permutation of summands, decomposition $\M = \oplus_{x\in S}\M_x$
into a direct sum of indecomposable $\C$-module categories.
\smallbreak

Let $\M$ be a right module category over $\C$. The {\em dual category} of $\C$ with respect to $\M$ is  the category $\C^*_\M:=\F_\C(\M,\M)$ whose objects are $\C$-module
functors from $\M$ to itself, and morphisms are natural
module transformations.
The category $\C^*_\M$ is a multi-fusion category with tensor product being composition
of module functors. It is known that if $\M$ is an indecomposable module category
over $\C$, then $\C^*_\M$ is a fusion category \cite{ENO}.

\subsection{Algebras in fusion categories}

Let $\C$ be a (strict) fusion category. An algebra $(A,\nabla,\eta)$
in $\C$ consists of an object $A$ and morphisms $\nabla:A\otimes
A\to A$, $\eta:\unit \to A $ such that \[\nabla(\nabla\otimes\id_A)=
\nabla(\id_A\otimes \nabla),\  \  \  \
\nabla(\eta\otimes\id_A)=\nabla(\id_A\otimes\eta)=\id_A.\]

Let $\M$ be a left $\C$-module category. A left $A$-module
$(M,\lambda)$ in $\M$ over an algebra $A$ in $\C$ is an object $M$
in $\M$ with a morphism $\lambda: A\otimes M\to M$ in $\M$ which is
associative  in the sense of
\[\lambda(\nabla\otimes\id_M)=\lambda(\id_A\otimes\lambda)m_{A,A,M}:
(A\otimes A)\otimes M\to M,\] and satisfies $\lambda
(\eta\otimes\id_M)=\id_M$. Morphisms are defined in the obvious way,
and the category of $A$-modules in $\M$ is denoted by $_A\M$.

In the same way we define the categories $\N_A$ of right $A$-modules
in $\N$, for a right $\C$-module, and the category  $_A\C_A$ of
$A$-bimodules in $\C$.

If $A$ is an algebra in $\C$, the category $\C_A$ is a left
$\C$-module category with action given by $\C\times \C_A\to \C_A,
(X,(M,\rho))\mapsto (X\otimes M, \id_X\otimes\rho).$

An algebra $A\in \C$ in a fusion category is called a \textit{semisimple algebra} if the category $\C_A$ is semisimple. If $A$ is semisimple then $_A\C$ and $_A\C_A$ are semisimple categories.

Let $\C$ be a fusion category, $A$ a semisimple algebra in $\C$, $M\in \C_A$, and $N\in _A\M$, where $\M$ is a left $\C$-module category. The tensor product $M\otimes_A N \in \M$ is defined as the coequalizer of

\[
\begindc{\commdiag}[50]
\obj(0,1)[aa]{$(M\otimes A)\otimes N$}
\obj(4,1)[bb]{$M\otimes N$}
\obj(6,1)[cc]{$M\otimes_A N$}
\mor{aa}{bb}{$(\id\otimes\lambda_N)m_{M,A,N}$}[+1,9]
\mor{aa}{bb}{$\lambda_{M}\otimes\id$}[-1,9]
\mor{bb}{cc}{}
\enddc
\]


A semisimple algebra $A$ is called indecomposable if $\C_A$ is an indecomposable left $\C$-module category. Under this condition, $(_A\C_A,\otimes_A, A)$ is a fusion category, see \cite{ENO}.

\subsection{Internal Hom and Morita theory}

An important technical tool in the study of module categories is the notion of internal $\Hom$. Let $\M$ be a module category over $\C$ and $M_1,M_2\in \M$. Consider the functor Hom$(- \otimes M_1,M_2)$ from the category $\C$ to the category of vector spaces. This functor is exact and thus is representable. The internal Hom $\underline{\text{Hom}}(M_1,M_2)$ is an object of $\C$ representing
the functor Hom$_\M(- \otimes M_1,M_2)$.  Given a non-zero object $M\in \M$, the internal Hom, $A=\underline{\text{Hom}}(M,M)\in \C$ has a natural algebra structure, such that the left $\C$-module category of right $A$-modules in $\C$ is equivalent to $\M$ as left $\C$-module categories, \cite[Theorem 1]{O1}.
\smallbreak
We recall a generalization of a theorem of Watts \cite{watts}, see \cite[Theorem 3.1]{action scha}.

\begin{teor}\label{teorema de watts}
Let $R$ and $S$ be semisimple algebras in a fusion category $\C$. The functor \[\mathcal{T}:\ _R\C_S \ni M\mapsto (-)\otimes _R M\in \F_\C(\C_R,\C_S)\] is a category equivalence. Its quasi-inverse equivalence maps a functor $F:\C_R \to\ \C_S$ to $\mathcal{T}^{-1}(F):= F(R)$, with the left $R$-module structure
\[
\begin{diagram}
  \node{R\otimes F(R) = F(R\otimes R)} \arrow{e,t}{F(\nabla)}\node{F(R)}
\end{diagram}
\]
\end{teor}
\qed

\begin{defin}
Let $\C$ be a fusion category, and $\M$ be a left $\C$-module
category. We shall say that an algebra $A$ in $\C$ represents $\M$
if $\C_A$ is equivalent to $\M$ as left $\C$-module categories.

If $A,B$ are algebras in $\C$, and $F:\C_A\to \C_B$ is a $\C$-module functor, we shall say that $\M\in _A\C_B$ represents $F$ if $F$ is equivalent to $(-)\otimes_A M$ as $\C$-module functors.

Two algebras $A,B$ in $\C$ are Morita equivalent if $\C_A\cong \C_B$
as $\C$-module categories.
\end{defin}

\begin{corol}\label{Morita}
Let $S$ and $R$ be semisimple algebras in $\C$, then they are Morita equivalent if and only if there exists an $(S,R)$-bimodule $M$ and an $(R,S)$-bimodule $N$, such that $M\otimes_B N \cong S$ as $S$-bimodule and $N\otimes_A M\cong R$ as $R$-bimodule.
\end{corol}
\qed

\subsection{Tensor products of module categories}

Recall that a monoidal category is called strict if the
associativity constraint is the identity. A module category
$(\M,\mu)$ over a strict tensor category is called a \emph{strict
module category}  if the constraint $\mu$ is the identity. By
\cite[Proposition 2.2]{G} we may assume that every fusion category
and every module category is strict.

\medbreak
Let $\C, \D$ be fusion categories. By definition, a $(\C,\D)$-bimodule category is a module category over $\C\boxtimes\D^{\text{rev}}$, where $\D^{\text{rev}}$ is the category $\D$ with reversed tensor product, and $\boxtimes$ is the  Deligne tensor product of abelian categories, see \cite{D}.

Let $\A$ be an abelian category and $\M, \N$ left and right (strict) $\C$-module categories, respectively.

\begin{defin}\cite[Definition 3.1]{ENO3}
Let $F :\M\times \N\to \A$ be a bifunctor exact in every
argument. We say that $F$ is $\C$-balanced if there is a natural family of
isomorphisms \[b_{M,X,N} : F(M \otimes X, N) \to
F(M, X \otimes N),\]
such that \[b_{M,X\otimes Y,N}= b_{M,X,Y \otimes N}\circ b_{M\otimes X,Y,N},\] for all $M\in \M, N\in \N, X,Y\in \C$.
\end{defin}

\begin{defin}\cite[Definition 3.3]{ENO3}
A tensor product of a right $\C$-module category $\M$ and
a left $\C$-module category $\N$ is an abelian category $\M\boxtimes_\C\N$ together
with a $\C$-balanced functor
 \[B_{\M,\N} :\M\times N \to M\boxtimes_\C \N\]
inducing, for every abelian category $\A$, an equivalence between the
category of $\C$-balanced functors from $\M\times\N$ to $\A$ and the category of
functors from $\M\boxtimes_\C \N$ to $\A$:
\[\F_{bal}(\M\times \N, \A) \cong \F(M\boxtimes_\C \N, \A).\]
\end{defin}

\begin{obs}
\begin{enumerate}
   \item  The existence of the tensor product for module categories over fusion categories was proved in \cite{ENO3}.
  \item If the modules categories are semisimple then the tensor product is a semisimple category, see \cite{ENO3}.
\end{enumerate}
\end{obs}

Given a right  $\C$-module functor $F : \M \to \M'$ and a left
$\C$-module functor $G : \N \to \N'$ note that $B_{\M',\N'}(F
\boxtimes G) : \M\boxtimes \N \to \M' \boxtimes_\C \N'$ is
$\C$-balanced. Thus the universality of $B$ implies the existence of
a unique right functor $F \boxtimes_\C G := \overline{B_{\M',\N'}(F
\boxtimes G)}$ making the diagram

\[
\begin{diagram}
\node{\M\boxtimes \N} \arrow{s,l}{B_{\M,\N}} \arrow{e,l}{F\boxtimes G} \node{\M'\boxtimes\N'} \arrow{s,r}{B_{\M',\N'}}\\
\node{\M\boxtimes_{\C}\N}\arrow{e,t}{F\boxtimes_\C G} \node{\M'\boxtimes_\C \N'}
\end{diagram}
\]commutative.
The bihomorphism $\boxtimes_\C$ is functorial in module functors, \emph{i.e.}, $(F'\boxtimes_\C E')(F\boxtimes_\C E)= F'F\boxtimes_\C E'E$.

\begin{rmks}\label{remarks sobre producto tensorial}

\begin{enumerate}
  \item If  $\M$ is a $(\C,\mathcal{E})$-bimodule category and $\N$ is an $(\mathcal{E},\D)$-bimodule
category, then $\M\boxtimes_\mathcal{E}\N$ is a $(\C,\D)$-bimodule category and $B_{M,N}$ is a $(\C,\D)$-bimodule functor, see \cite[Proposition 3.13]{Justin}.
  \item Let $\M$ be a $(\C,\D)$-bimodule category. The $\C$-module action in $\M$ is balanced. Let $r_\M :\C\boxtimes_\C\M \to \M$ denote the unique functor factoring through $B_{\M,\D}$. In \cite[Proposition 3.15]{Justin} it was proved that $r_\M$ is a $(\C,\D)$-module equivalence.
  \item Let $\M$ be a right $\C$-module category, $\N$ a $(\C,\D)$-bimodule category,
and $\mathcal{K}$ a left $\D$-module category. Then by \cite[Proposition 3.15]{Justin}, there is a canonical equivalence $(\M\boxtimes_\C \N)\boxtimes_\mathcal{D} \mathcal{K} \cong
\M\boxtimes_\C (\N \boxtimes\D \mathcal{K})$ of bimodule categories. Hence the notation $M\boxtimes_\C \N \boxtimes_\D \mathcal{K}$ will yield no ambiguity.
\end{enumerate}
\end{rmks}

\subsection{Graded fusion categories}
Let $\C$ be a fusion category and let $G$ be a finite group. We say that $\C$ is graded by
$G$ if $\C =\oplus_{\sigma\in G} \C_\sigma$, and for any $\sigma, \tau \in G$, one has $\otimes: \C_\sigma \times \C_\tau \to \C_{\sigma\tau}$.
\smallbreak
The fusion subcategory $\C_e$ corresponding to the neutral element $e \in G$ is called the trivial component of the $G$-graded category $\C$. A grading is faithful if $\C_\sigma \neq 0$ for all $\sigma \in G$, in this paper we shall consider only faithful gradings.

\subsection{Graded module categories}

\begin{defin} Let $\C=\oplus_{\sigma\in G}\C_\sigma$ be a graded fusion category and let $X$ be a left
$G$-set. A  \emph{left $X$-graded $\C$-module category}  is a left $\C$-module category $\M$ endowed
with a decomposition  \[\M=\oplus_{x\in X}\M_x,\] into a direct sum of full abelian subcategories,
such that for all $\sigma\in G$, $x\in X$, the bifunctor  $\otimes$ maps $\C_\sigma\times \M_x$ to
$\M_{\sigma x}$. \end{defin} An  \emph{$X$-graded $\C$-module functor}  $F:\M\to \N$ is a $\C$-module
functor such that $F(\M_x)$ is mapped to $\N_x$, for all $x\in X$.

\section{Induced module categories}

Let $\C$ be a fusion category and let $\D\subset \C$ be a fusion subcategory. Given a left $\D$-module category $\M$ and an algebra $A$ in $\D$, such that $\M\cong \D_A$ as left $\D$-module categories, we define the \emph{induced module category} from $\D$ to $\C$, denoted by $\ind_\D^\C(\M)$, as the $\C$-module category $\C_A$.

If  $\M= \D_A$ and $\N=\D_B$ are $\D$-module categories and  $F:\M\to \N$ is a $\D$-module functor represented by a $(B,A)$-bimodule $M$, then $M$ defines a $\C$-module functor $\ind_\D^\C(F)=(-)\otimes_AM: \ind_\D^\C(\M)\to \ind_\D^\C(\N)$, called the \textit{induced module functor}.

\begin{obs}
\begin{enumerate}
  \item By Corollary \ref{Morita}, the equivalence class of Ind$_\D^\C(\M)$ as $\C$-module category  does not depend on the algebra $A$.

  \item If $\C''\subset \C'\subset \C$ are fusion categories and $\M$ is a $\C''$-module category then $\ind_{\C''}^{\C'}(\ind_{\C'}^{\C}(\M))=\ind_{\C''}^{\C}(\M)$. This is analogous for functor and natural transformations of module categories.
\end{enumerate}
\end{obs}

Given a $G$-graded fusion category $\C$, and a subgroup $H\subset G$, we shall denote by $\C_H$ the fusion subcategory $\oplus_{h\in H}\C_h$.

\begin{prop}
For each subgroup $H\subset G, $ induction defines a 2-functor from the 2-category of $\C_H$-module categories to the 2-category of $\C$-module categories graded over $G/H$.
\end{prop}
\begin{proof}
Let $\M$ be a $\C_H$-module category, and suppose that $\M= (\C_H)_{A}$ for some algebra $A\in \C_H$. If $M\in \C_A$, then $M=\bigoplus_{gH\in G/G} M_{gH}$, where $M_{gH}\in \C_{gH}$. Thus $M_{gH}\in\C_A$, so $\ind_{\C_H}^\C(\M)=\C_A$ is graded by $G/H$.

Let $\M$ and $\N$ be $\C_H$-module categories such that $\M=(\C_H)_A$ and $\N=(\C_H)_B$ for algebras $A, B\in \C_H$. Every $\C_H$-module functor from $\M$ to $\N$ is equivalent to the functor $(-)\otimes_A M $ for some $(A$-$B)$-bimodule in $\C_H$ (see Theorem \ref{teorema de watts}). Then the induced functor from $\C_A$ to $\C_B$ is $(-)\otimes_A M $, and it is $G/H$-graded.
\end{proof}

\begin{teor}\label{2-equivalencia con graduados}
Let $\C=\oplus_{\sigma\in G}\C_\sigma$ be a $G$-graded fusion
category. For each $H\subset G$, the induction of module categories
and module functor defines a 2-equivalence  from the 2-category of
module categories over $\C_H$ and $\C$-module categories
$G/H$-graded.
\end{teor}
\begin{proof}
Let $\M$ be a $\C$-module category $G/H$-graded. Then $\M_H$ is a
$\C_H$-module category. For each $0\neq M \in \M_H$, the internal
Hom, $\underline{\text{Hom}}(M,M)\in\C_H$ with respect to the module
category $\M_x$ is also the internal Hom in $\C$ of $\M$. In fact,
if $X\in \C_\sigma$, and  $\sigma\notin H$, then $X\otimes M\notin
\M_x$, so
\begin{align*}
\Hom_\M(X\otimes M,M)= 0= \Hom_\C(X,\underline{\Hom}(M,M)).
\end{align*}
Then $\M\cong \C_{\underline{\Hom}(M,M)} = $ Ind$_{\C_H}^\C(\M_x)$.

Let $\M$ and $\N$ be $\C$-module categories graded over $G/H$ such
that $\M=\C_A$ and $\N=\C_B$ for algebras $A, B\in \C_H$. Then by
Theorem \ref{teorema de watts} every $\C$-module functor from $\M$
to $\N$ is equivalent to the functor $(-)\otimes_A M $ for some
$(A$-$B)$-bimodule in $\C$. If $M$ defines a $G/H$-module functor
then $M\in \C_H$, so by definition is the induced of a $\C_H$-module
functor.
\end{proof}

\begin{obs}
By  Theorem \ref{2-equivalencia con graduados} the 2-category of
module categories over $\C_e$ and the 2-category of $G$-graded
$\C$-module categories are 2-equivalent.
\end{obs}

\begin{prop}\label{Prop Inducida como producto tensorial}
Let $\C$ be a $G$-graded fusion category. If $\M$ is a left
$\C_H$-module category then $\C\boxtimes_{\C_H}\M \cong
\ind_{\C_H}^\C(\M)$ as $\C$-module categories.
\end{prop}
\begin{proof}
The $\C$-module category $\C\boxtimes_{\C_H}\M$ is $G/H$-graded
where $(\C\boxtimes_{\C_H}\M)_{gH}= (\C_{gH})\boxtimes_{\C_H}\M$. In
particular $(\C\boxtimes_{\C_H}\M)_{H}= \C_H\boxtimes_{\C_H}\M\cong
\M$ as $\C_H$-module categories, then by Theorem \ref{2-equivalencia
con graduados}, $\C\boxtimes_{\C_H}\M\cong \ind_{\C_H}^\C(\M)$.
\end{proof}

\begin{obs}
Proposition \ref{Prop Inducida como producto tensorial} implies that
$\ind_{\C_H}^\C(\M)$ is a semisimple $\C$-module category. In fact,
by \cite[Proposition 3.5]{ENO3}, the category $\C\boxtimes_{\C_H}\M$
is equivalent to $\Hom_{\C_H}(\C^{op}, \M)$, and by \cite[Theorem
2.16]{ENO}, the category $\Hom_{\C_H}(\C^{op}, \M)$ is semisimple.
\end{obs}

Let $\M=\bigoplus_{gH\in G/H}\M_{gH}$ be a $G/H$-graded $\C$-module
category. We define the canonical functor
$\overline{\mu}:\C\boxtimes \M_{H}\to \M, X\boxtimes M\mapsto
X\otimes M$. It is clear that $\overline{\mu}$ is a $\C_H$-balanced
functor, thus $\overline{\mu}$ defines a $\C$-module functor
$\mu:\C\boxtimes_{\C_H}\M_{H}\to \M$.

\begin{prop}\label{equivalencia mu}
The functor $\mu:\C\boxtimes_{\C_H}\M_{H}\to \M$ is a $G/H$-graded equivalence of $\C$-module categories. In particular, each $\mu_{gH}:\C_{gH}\boxtimes_{\C_H}\M_{H}\to \M_{gH}$ is an equivalence of left $\C_H$-module categories.
\end{prop}
\begin{proof}
The functor $\mu$ is a $G/H$-graded functor, and $\mu_{H}=r_\M:\C_H\boxtimes_{\C_H}\M\to \M$ is an equivalence of $\C_H$-module categories (see part 2 of Remark \ref{remarks sobre producto tensorial}). Thus by Theorem \ref{2-equivalencia con graduados}, $\mu\cong \ind(\mu_H)$  as $\C$-module functors, and it is an equivalence.
 \end{proof}

The following result appears in \cite[Theorem 6.1]{ENO3} we provide
an alternate proof using our results.

\begin{corol}\label{mapa picard}
For every  $\sigma, \tau \in G$, the canonical functor
\[\mu_{\sigma,\tau}: \C_\sigma\boxtimes_{\C_e}\C_\tau\to
\C_{\sigma\tau},\] is an equivalence of $\C_e$-bimodule categories.
\end{corol}

\begin{proof} Let us consider the graded  $\C$-module category $\C(\tau)$, where $\C=\C(\tau)$ as
$\C$-module categories, but with grading $(\C(\tau))_\sigma= \C_{\tau \sigma}$, for $\tau\in G$.

Since $\C(\tau)_e=\C_\tau$, by Proposition \ref{equivalencia mu}, the canonical functor $\mu_\sigma:\C_\sigma \boxtimes_{\C_e}\C_\tau \to \C(\tau)_\sigma= \C_{\tau\sigma}$ is a
$\C_e$-bimodule category equivalence. But by definition $\mu_\sigma= \mu_{\sigma,\tau}$, so the proof is finished.
\end{proof}

\section{Clifford theory for graded fusion categories}\label{section clifford theory}
In this section we shall denote by $\C$ a fusion category graded by a finite group $G$.

Let $\M$ be a $\C$-module category, and let $\N \subset\M$ be a full abelian subcategory. We shall denote by $\C_\sigma\overline{\otimes} \N $ the full abelian subcategory given by $Ob(\C\overline{\otimes}
\N)$ = $\{$subquotients of $V\otimes N : V \in \C_\sigma, N \in
N\}$. (Recall that a subquotient object is a subobject of a quotient object.)

Let $\M$ be a $\C$-module category and let $\N$ be a $\C_e$-submodule category
of $\M$. The bifunctor $\otimes$ induces a canonical $\C_e$-module functor  $\mu_\sigma: \C_\sigma\boxtimes_{\C_e}\N\to \C_\sigma\overline{\otimes}\N$.
\begin{prop}\label{equivalencia sub}
Let $\M$ be a $\C$-module category and let $\N$ be an indecomposable $\C_e$-submodule category
of $\M$. Then the canonical $\C_e$-module functor $\mu_\sigma: \C_\sigma\boxtimes_{\C_e} \N \to \C_\sigma\overline{\otimes}
\N$, is an equivalence of  $\C_e$-module categories, for all $\sigma\in G$.
\end{prop}
\begin{proof}
Define a $G$-graded $\C$-module category by $gr-\N = \bigoplus_{\sigma\in G} \C_\sigma\overline{\otimes}\N$, with action
\begin{align*}
\otimes: \C_\sigma\times \C_h\overline{\otimes}\N &\to \C_{\sigma h}\overline{\otimes}\N\\
V_\sigma\times T &\mapsto V_\sigma\otimes T.
\end{align*}
Since $\N$ is indecomposable,  $\C_e\overline{\otimes}\N = \N$ as a
$\C_e$-module category, so by  Proposition \ref{equivalencia mu}
the canonical functor $\mu : \C\boxtimes_{\C_e}\N \to gr-\N$ is a
category equivalence of $G$-graded $\C$-module categories and the
restriction $\mu_\sigma: \C_\sigma\boxtimes_{\C_e} \N \to
\C_\sigma\overline{\otimes}\N$ is a $\C_e$-module category
equivalence.
\end{proof}

Given a $\C$-module category $\M$, we shall denote by
$\Omega_{\C_e}(M)$ the set of equivalence classes of indecomposable
$\C_e$-submodule categories of $\M$.
\begin{corol}\label{G-conjuto} Let $\M$ be a $\C$-module category. The group $G$ acts on $\Omega_{\C_e}(\M)$ by
\begin{align*}
    G\times \Omega_{\C_e}(\M)&\to \Omega_{\C_e}(\M)\\
    (g,[\N]) &\mapsto [\C_\sigma\boxtimes_{\C_e}\N]
\end{align*} \end{corol}

\begin{proof}

Let $\N$ be an indecomposable  $\C_e$-submodule category of $\M$. By Proposition \ref{equivalencia sub} the
functor \begin{align*}
    \mu_\sigma: \C_\sigma\boxtimes_{\C_e} \N&\to \C_\sigma\overline{\otimes} \N
\end{align*}is a $\C_e$-module category equivalence, so  $\C_\sigma\boxtimes_{\C_e} \N$ is equivalent
to a  $\C_e$-submodule category of $\M$. \end{proof} Let $\M$ be an
abelian category and let  $\N, \N'$  be full abelian subcategories
of $\M$, we shall denote by $\N +\N'$ the full abelian subcategory
of $\M$ where $Ob(\N +\N')=\{ \text{subquotients  of } N \oplus N'
:N\in \N, N'\in \N' \}$. It will be called the sum category of $\N$
and $\N'$. \smallbreak Now, we are ready to prove our first main
result:
\begin{proof}[Proof of Theorem \ref{theorem clifford}]
(1) \  Let $\N$ be an indecomposable $\C_e$-submodule category of $\M$, the canonical functor
\begin{align*}
    \mu: \C\boxtimes_{\C_e}\N \to \M
\end{align*} is a $\C$-module functor and $\mu=\oplus_{\sigma\in G} \mu_\sigma$, where
$\mu_{\sigma}=\mu|_{\C_\sigma\boxtimes_{\C_e}\N}.$ By  Proposition \ref{equivalencia sub} each $\mu_\sigma$ is  a
$\C_e$-module category equivalence with  $\C_\sigma\overline{\otimes}\N$.

Since $\M$ is indecomposable, every object  $M\in \M$ is isomorphic to some subquotient of $\mu(X)$ for some object $X\in
\C\boxtimes_{\C_e}\N$. Then   $\M=\sum_{\sigma\in G}\C_\sigma\overline{\otimes} \N$, and each
$\C_\sigma\overline{\otimes} \N$ is an indecomposable  $\C_e$-submodule category.

Let $S, S'$ be indecomposable $\C_e$-submodule categories of $\M$. Then there exist  $\sigma, \tau \in G$ such that $\C_\sigma\boxtimes_{\C_e}\N\cong S$, $\C_\tau\boxtimes_{\C_e}\N\cong S'$, and by
Corollary \ref{mapa picard}, $S'\cong \C_{\tau\sigma^{-1}}\boxtimes_{\C_e}S$. So the action is
transitive. \bigbreak

(2) Let  $H=\st([\N])$ be the stabilizer subgroup of  $[\N]\in \Omega_{\C_e}(\M)$ and
\[\M_\N=\sum_{h\in H} \C_h\overline{\otimes} \N.\]

Since  $H$ acts transitively on $\Omega_{\C_e}(\M_\N)$, the $\C_H$-module category $\M_\N$ is indecomposable.
Let $\Sigma= \{e,\sigma_1,\ldots,\sigma_n\}$ be a set of representatives of the cosets of  $G$ modulo $H$. The
map $\phi: G/H\to \Omega_{\C_H}(\M)$, $\phi(\sigma H)= [\C_\sigma\overline{\otimes}\M_\N]$ is an
isomorphism of $G$-sets. Then $\M$ has a structure of $G/H$-graded $\C$-module category, where $\M=$
$\oplus_{\sigma \in \Sigma}\C_\sigma\overline{\otimes} \M_\N$. By Theorem \ref{2-equivalencia con graduados}, $\M\cong \C\boxtimes_{\C_H}\M_\N$ as $\C$-module categories.
\end{proof}

\begin{defin}
Let $\C$ be a $G$-graded fusion category. If $(\M,\otimes)$ is a
$\C_e$-module category, then a $\C$-extension of $\M$ is a
$\C$-module category $(\M,\odot)$ such that $(\M,\otimes)$ is
obtained by restriction to $\C_e$.
\end{defin}

\begin{corol}\label{corol main result}
Let $\M$ be an indecomposable $\C$-category, and $\N$ an
indecomposable $\C_e$-module category. Then there exists a subgroup
$S\subset G$, and a $\C_S$-extension $(\N,\odot)$ of $\N$, such that
$\M\cong \C\boxtimes_{\C_S}\N$ as $\C$-module categories.
\end{corol}
\begin{proof}
Let $\N\subset \M$ be an indecomposable $\C_e$-submodule category.
By the proof of part (1) of Theorem \ref{theorem clifford}, we have
that $\M=\sum_{\sigma\in G}\C_\sigma\overline{\otimes}\N$, where
each $\C_\sigma\overline{\otimes}\N$ is an indecomposable
$\C_e$-module subcategory of $\M$. Let $X=G/\sim$, where $\sigma\sim
\tau$ if and only if $\C_\sigma\overline{\otimes}\N
=\C_{\tau}\overline{\otimes}\N$ for $\sigma, \tau \in G$. Then $\M=
\bigoplus_{x\in X}\C_x$ as a direct sum of $\C_e$-module categories,
and $\M$ is an $X$-graded $\C$-module category, where $X$ has the
following $G$-action: \[\sigma y=x  \text{ if and only if }
\C_\sigma\overline{\otimes}\C_y\overline{\otimes}\N
=\C_x\overline{\otimes}\N,\] for $\sigma \in G, x,y \in X$.

Again by part (1) of Theorem \ref{theorem clifford}, $X$ is a transitive $G$-set. Then by Theorem \ref{2-equivalencia con graduados}, $\M\cong \C\boxtimes_{S}\N$ as $\C$-module categories, where $S=\{\sigma\in G| \C_\sigma\overline{\otimes}\N =\N\}$.
\end{proof}

\begin{obs}\label{obs iso de g-set}
Following with the notations in the proof of Corollary \ref{corol main result}, if $\M$ is an indecomposable $\C$-module category,  the set of indecomposable $\C_e$-module subcategories of $\M$ is a transitive $G$-set isomorphic to $G/S$.
\end{obs}

\begin{prop}\label{prop equivalent inducidas}
Let $H,H'\subset G$ be subgroups, and $(\N,\odot)$, $(\N',\odot')$
be a $\C_H$-extension and a $\C_{H'}$-extension of the
indecomposable $\C_e$-module categories $\N$ and $\N'$,
respectively. Then $\C\boxtimes_{\C_{H'}}\N' \cong
\C\boxtimes_{\C_H}\N$ if and only if there exists $\sigma \in G$
such that $H=\sigma H'\sigma^{-1}$ and $\C_{\sigma
H'}\boxtimes_{\C_{ H'}}\N'\cong \N$ as $\C_H$-module categories.
\end{prop}
\begin{proof}
Let $F:\C\boxtimes_{\C_{H'}}\N' \to \C\boxtimes_{\C_H}\N$ be an
equivalence of $\C$-module categories. The functor $F$ maps
indecomposable $\C_e$-module subcategories of
$\C\boxtimes_{\C_{H'}}\N'$ to indecomposable $\C_e$-module
subcategories of $\C\boxtimes_{\C_H}\N$, and $F(X\otimes V)\cong
X\otimes F(V)$ for all $X\in \C$, $X\in \C\boxtimes_{\C_{H'}}\N'$,
thus  (see Remark \ref{obs iso de g-set}) $F$ defines a $G$-set
isomorphism from $G/H$ to $G/H'$. Recall that  $G/H$ and $G/H'$ are
isomorphic if and only if there exists $\sigma \in G$ such that
$H=\sigma H'\sigma^{-1}$, and the $G$-isomorphism is defined by
$\sigma H'\mapsto H$. Then, by restriction, the functor $F$ defines
an equivalence of $\C_H$-module categories $F:\C_{\sigma
H'}\boxtimes_{\C_{ H'}}\N'\to \C_{H}\boxtimes_{\C_H}\N\cong \N$.
\end{proof}
\section{Invariant and $G$-graded module categories}\label{Section invariant y G graduado}

In this section we shall denote by $\C$ a $G$-graded fusion category.

\begin{defin}
A $\C_e$-module category $\M$ is called $G$-invariant if $\C_\sigma\boxtimes_{\C_e}\M$ is equivalent to $\M$ as $\C_e$-module categories, for all $\sigma\in G$.
\end{defin}

By Corollary \ref{corol main result}, every $\C$-module category is
equivalent to an induced module category of an $S$-invariant
$\C_S$-module category for some subgroup $S\subset G$. \smallbreak
Let $\C$ be a $G$-graded tensor category. For a $G$-graded
$\C$-module category $\M=\bigoplus_{\sigma\in G}\M_\sigma$ and
$\sigma \in G$, we define a new $G$-graded $\C$-module category
$\M(\sigma)$, as the $\C$-module category $\M$ with
$\M(\sigma)_\tau= \M_{\tau\sigma}$ for all $\tau \in G$.

\smallbreak A $\C$-module functor $F: \M\to \N$ is called a graded
$\C$-module functor of degree $\sigma$ ($\sigma \in G$), if
$F(M_x)\in \N_{x\sigma}$ for all $x\in G, M_x\in \M_x$. Graded
module functors of degree $\sigma$ build a full abelian subcategory
$\F_\C(\M,\N)_\sigma$ of $\F_\C(\M,\N)$. We also have  following
equalities:

\begin{align*}
    \F_\C(\M,\N)_e &= \F^{gr}_\C(\M,\N),\\
    \F_\C(\M,\N)_\sigma &= \F_\C^{gr}(\M,\N(\sigma))= \F_\C^{gr}(\M(\sigma^{-1}),\N).
\end{align*}

Note that  $\C_\M^*=\F_\C(\M,\M) = \bigoplus_{\sigma \in
G}\F_\C(\M,\M)_{\sigma}$, and with this grading $\C_\M^*$ is a
$G^{op}$-graded (multi)-fusion category.
\begin{defin}
A graded tensor category over a group $G$ will be called a crossed
product tensor category if every homogeneous component has at least
one multiplicatively invertible object.
\end{defin}

\begin{prop}\label{prop  invariant sii product cruzado}
Let $\C$ be a $G$-graded fusion category. An inde\-com\-po\-sable $\C_e$-module category $\N$ is invariant if and only if $\C_{\C\boxtimes_{\C_e}\N}^*$ is $G^{op}$-crossed product fusion category.
\end{prop}
\begin{proof}
By Theorem \ref{2-equivalencia con graduados}, if $\M$ and $\N$ are $G$-graded $\C$-module category, the induction functor  defines  equivalences \[ \F_\C(\M,\N)_\sigma \cong  \F_{\C_e}(\M_e,\N_\sigma) \cong  \F_{\C_e}(\M_{\sigma^{-1}},\N_e).\]

Thus  $\F_\C(\C\boxtimes_{\C_e}\N,\C\boxtimes_{\C_e}\N)$ is a crossed product if and only if $\F_\C(\C\boxtimes_{\C_e}\N,\C\boxtimes_{\C_e}\N)_\sigma \cong \F_{\C_e}(\N,\C_\sigma\boxtimes_{\C_e}\N)$ has at least one equivalence for all $\sigma\in G$, \emph{i.e.} if $\N$ is $G$-invertible.
\end{proof}
Let $\C$ be a fusion category. An indecomposable $\C$-module category is called pointed module category if $\C_\M^*$ is a pointed fusion category. A fusion category is called \emph{group-theoretical} if it admits a pointed $\C$-module category, see \cite{ENO}. Group-theoretical categories can be explicitly described in terms of finite groups and their cohomology (see \cite{O2}).

We are now ready to prove Theorem \ref{teor group-theo}, which is a generalization of \cite[Theorem 3.5]{Non-grouptheo}

\begin{proof}[Proof of Theorem \ref{teor group-theo}]
Let $\N$ be a $G$-invariant pointed $\C_e$-module category. Then by
Proposition \ref{prop  invariant sii product cruzado},
$\C_{\C\boxtimes_{\C_e}\N}^*$  is  a $G^{op}$-crossed product
category, where $(\C_{\C\boxtimes_{\C_e}\N}^*)_e\cong (\C_e)_\N^*$
is a pointed fusion category, so $\C_{\C\boxtimes_{\C_e}\N}^*$ is
pointed.

Conversely, let $\M$ be a pointed $\C$-module category and
$\N\subseteq \M$ be a $\C_e$-submodule category. Then by \cite[Lemma
2.2]{Non-grouptheo} and Proposition \ref{equivalencia mu} the
stabilizer of $[\N]$ in $\Omega_{\C_e}(\M)$ is $G$, so $\N$ is
$G$-invariant. Now, the same argument of  \cite[Theorem
3.5]{Non-grouptheo} implies that $\N$ is a pointed $\C_e$-module
category.
\end{proof}

\begin{obs}
 Let $\C$ be a $G$-graded fusion category. By
\cite[Theorem 3.3]{GNN} and \cite[Remark 2.11]{GNN}, the group $G$
acts on $\mathcal{Z}(\C_e)$ by braided autoequivalences, and
therefore $G$ also acts on the set of Lagrangian subcategories of
$\mathcal{Z}(\C_e)$ (see \cite[subsection 1.4.1]{DGNO} for the
definition of Lagrangian subcategories of a braided fusion
category). In \cite[Corollary 3.10]{GNN} the following criterion for
a graded fusion category to be group-theoretical was proved: $\C$ is
group-theoretical if and only if $\mathcal{Z}(\C_e)$ contains a
$G$-stable Lagrangian subcategory. As we shall see, the above
criterion is equivalent to the criterion Theorem \ref{teor
group-theo}. Using the bijective correspondence between equivalence
classes of Lagrangian subcategory of $\mathcal{Z}(\C_e)$ and pointed
$\C_e$-module categories proved in \cite[Theorem 4.66]{DGNO}, and
the description of the action of $G$ on $\mathcal{Z}(\C_e)$ given in
\cite[subsection 3.1]{GNN}, it is easy to see that if $\N$ is a
pointed $\C_e$-module category and $\mathcal{L}$ is the
corresponding Lagrangian subcategory of $\mathcal{Z}(\C_e)$, then
for any $\sigma\in G$, the Lagrangian subcategory corresponding to
$\C_\sigma\boxtimes_{\C_e} \N$ is $\sigma_*(\mathcal{L})$. Thus,
$\C_e$ has a $G$-invariant $\C_e$-module category if and only if
$\mathcal{Z}(\C_e)$ has a $G$-stable Lagrangian subcategory.
\end{obs}

\begin{corol}\label{corol que aplica a TY}
Let $\C$ be a $G$-graded fusion category such that $\C_e \cong
\text{Vec}_A$, where $A$ is an abelian group. Then $\C$ has an
invariant module category if and only if $\C$ is group-theoretical.
\end{corol}
\begin{proof}
If $A$ is an abelian group, then every indecomposable module
category over Vec$_A$ is pointed, see \cite[Theorem 3.4]{Naidu}. So,
the corollary follows immediately by Theorem \ref{teor group-theo}.
\end{proof}

A Tambara-Yamagami category is a $\mathbb Z_2$-graded fusion category $\C=\C_0\oplus\C_1$, such that $\C_0$ is pointed and $\C_1$ has only one simple object up to isomorphisms, see \cite{TY} for a complete classification.
\begin{obs}
Corollary \ref{corol que aplica a TY}\ applies to Tambara-Yamagami
categories, and their generalizations  \cite{2Gen-TY},
\cite{Gen-TY}. In particular, if a fusion category of this type is
the category of representations of a Hopf algebra, \emph{i.e.}, if
it has a fiber functor, then it is a group-theoretical fusion
category (it is well-known for TY categories, \cite[Remark
8.48]{ENO}).
\end{obs}

Module categories over group-theoretical fusion categories were
classified by Ostrik in \cite{O2}. As an application of some of our
results, we shall describe indecomposable module categories over a
non group-theoretical Tamabara-Yamagami category.

\begin{prop}
Let $\C$ be a non group-theoretical Tambara-Yamagami ca\-tegory,
where $\C_0=\text{Vec}_A$ for an abelian group $A$. Then every
indecomposable $\C$-module category is equivalent to $\C_{k_\alpha
B}$, where $k_\alpha B\in $Vec$_A\subset \C$ is a twisted group
algebra for $B\subset A$.
\end{prop}
\begin{proof}

By Corollary \ref{corol que aplica a TY}, a Tambara-Yamagami
category is non group-theo\-retical if and only if it has only non
$\mathbb Z_2$-invariant $\C_0$-module categories. Let $\M$ be an
indecomposable $\C$-module category, and $\N\subset \M$ be an
indecomposable $\C_0$-module category. If $\N$ is an
indecompo\-sable module category over $\C_0=$Vec$_A$, then by
\cite{O2}, there exists a subgroup $B\subset A$, and $\alpha\in
Z^2(B,k^*)$ such that $\N\cong (\C_0)_{k_\alpha B}$ as $\C_0$-module
categories.

Since  $\C$ has only non $\mathbb Z_2$-invariant $\C_0$-module
categories, Theorem \ref{theorem clifford} implies that $\M\cong
\ind^\C_{\C_0}\N= \C_{k_\alpha B}$.
\end{proof}

\section{Extending indecomposable module categories}

Corollary \ref{corol main result} reduces the construction of indecomposable $\C$-module categories over a graded fusion category $\C=\bigoplus_{\sigma\in G}\C_\sigma$, to the construction  of  $\C_H$-module categories $\M$ such that the restriction to $\C_e$  remains indecomposable, for some subgroup $H\subset G$.

In this section we shall provide a necessary and sufficient
condition when an indecomposable $\C_e$-module category can be
extended.

\subsection{Semi-direct product and equivariant fusion categories}

Let $\underline{\text{Aut}_\otimes(\C)}$ be the monoidal category where objects are tensor auto-equivalences of $\C$, arrows are tensor natural isomorphisms, and tensor product is the composition of functors. An action of the group  $G$ over a  monoidal category  $\C$, is a
monoidal functor  $*:\underline{G}\to
\underline{\text{Aut}_\otimes(\C)}$.

Given an action $*:\underline{G}\to
\underline{\text{Aut}_\otimes(\C)}$ of $G$ on $\C$,  the semi-direct
product fusion category, denoted by $\C\rtimes G$ is defined as
follows: As an abelian category $\C\rtimes G= \bigoplus_{\sigma\in
G}\C_\sigma$, where $\C_\sigma =\C$ as an abelian category, the
tensor product is \[[X, \sigma]\otimes [Y,\tau]:= [X\otimes
\sigma_*(Y), \sigma\tau],\ \  \   X,Y\in \C,\ \ \sigma,\tau\in G,\]
and the unit object is $[\unit,e]$. See \cite{tambara} for the
associativity constraint and a proof of the pentagon identity.

The category  $\C\rtimes G$ is $G$-graded by  \[\C\rtimes G=\bigoplus_{\sigma\in G}(\C\rtimes
G)_\sigma, \ \  \  \text{where}\  \  (\C\rtimes G)_\sigma = \C_\sigma,\] and the objects $[\unit,
\sigma]\in (\C\rtimes G)_\sigma$ are invertible, with inverse $[\unit,\sigma^{-1}]\in (\C\rtimes
G)_{\sigma^{-1}}$.

Another useful construction of a fusion category starting from a
$G$-action over a fusion category $\C$ is  the
$G$-equivariantization of $\C$, denoted by $\C^G$, see \cite{DGNO}.
Objects of this category are objects $X$ of $\C$ equipped with an
isomorphism $u_\sigma : \sigma_*(X)\to X$ for all $\sigma \in G$,
such that
$$u_{\sigma\tau}\circ\gamma_{\sigma,\tau} = u_\sigma\circ\sigma_*(u_\tau),$$
where $\gamma_{\sigma,\tau}: \sigma_*(\tau_*(X)) \to (\sigma\tau)_*(X)$ is the natural isomorphism associated to the action.
Morphisms and tensor product of equivariant objects are defined in an obvious
way.

\begin{lem}\label{lema semidir sii funtor graduado}
Let $\C$ be a $G$-graded fusion category, then:

\begin{enumerate}
  \item $\C$ is equivalent to a semi-direct product fusion category over $G$ if and only if there is a $G$-graded tensor functor $\text{Vec}_G\to \C$.
  \item There exists a correspondence between $G$-actions over $\C_e$, such that the associated semidirect product is tensor equivalent to $\C$, and  $G$-graded tensor functors $\text{Vec}_G\to \C$.
\end{enumerate}

\end{lem}
\begin{proof}
It follows from \cite[Section 3]{G}.
\end{proof}

\subsection{Extending indecomposable module categories}

Recall, if $(\M,\otimes)$ is a $\C_e$-module category, then an
extension of $\M$ is a $\C$-module category $(\M,\odot)$ such that
$(\M,\otimes)$ is obtained by restriction to $\C_e$. We shall say
that two $\C$-extension of a $\C_e$-module categories are equivalent
if they are equivalent as  $\C$-module categories.

Each graded functor $(F,\mu):$Vec$_G\to \C$ defines a $\C$-extension $(\C_e,\odot^F,m^F)$ of the $\C_e$-module category $(\C_e, \otimes)$ in the following way: Let $F(k_\sigma)= \unit_\sigma,$ and $\mu_{\sigma, h}:\unit_{\sigma h}\to \unit_\sigma\otimes \unit_h$, then we define $(\C_e,\odot^F,m^F)$ by
\[
V_\sigma\odot^F X_e = V_\sigma\otimes X_e\otimes \unit_{\sigma^{-1}}, \  \
m_{V_\sigma,V_\tau, X_e}^F= \id_{V_\sigma\otimes V_\tau\otimes X_e}\otimes \gamma_{k_{\tau^{-1}},k_{\sigma^{-1}}}\] for all $V_\sigma \in \C_\sigma, V_\tau \in\C_\tau, X_e\in \C_e,$ and  $\sigma,\tau \in G$.

\begin{defin}
We shall say that two graded functors $(F,\gamma), (F',\gamma'):
\text{Vec}_G\to \C$ are conjugate if there is a multiplicatively
invertible object $U\in \C_e$ and a family of isomorphisms
$\Theta_\sigma:F(k_\sigma)\otimes U\to U\otimes F'(k_\sigma)$ for
all $\sigma \in G,$ such that the diagram

\begin{equation}\label{pentagono conjugados}
\begin{diagram}
\node{F(k_{\sigma\tau})\otimes U} \arrow{s,l}{\gamma_{k_\sigma,k_\tau}\otimes\id_U}\arrow[2]{e,t}{\Theta_{\sigma\tau}} \node{} \node{U\otimes F'(k_{\sigma\tau})} \arrow{s,r}{\id_U\otimes\gamma'_{k_\sigma,k_\tau}}\\
\node{F(k_\sigma)\otimes F(k_\tau)\otimes U}\arrow{se,b}{\id_{F(k_\sigma)}\otimes \Theta_\tau} \node{}\node{U\otimes F'(K_\sigma)\otimes F'(k_\tau)}\\
\node{}\node{F(k_\sigma)\otimes U\otimes F'(k_\tau)} \arrow{ne,r}{\Theta_{\sigma}\otimes\id_{F'(k_\tau)}}
\end{diagram}
\end{equation}commutes for all $\sigma,\tau \in G$.
\end{defin}

If $(F,\gamma), (F',\gamma'): \text{Vec}_G\to \C$ are conjugate
graded functors, then the $\C$-extensions associated are equivalent.
In fact, the functor $\C_e\to \C_e, X_e\mapsto X_e\otimes U$, with
the natural isomorphisms \[\id_{V_\sigma\otimes X_e}\otimes
\Theta_{\sigma^{-1}}: V_\sigma\otimes X_e\otimes
F(k_{\sigma^{-1}})\otimes U \to V_\sigma\otimes X_e\otimes U\otimes
F'(k_{\sigma^{-1}})\]for all $V_\sigma \in \C_\sigma, X_e\in \C_e$,
define a $\C$-module equivalence.
\begin{prop}\label{lema extension facil}
Let $\C$ be a graded fusion category. Then the $\C_e$-module category $(\C_e,\otimes)$ can be extended if and only if $\C$ is a semi-direct product. There is a one-to-one correspondence between equivalence classes of $\C$-extensions of $\C_e$ and conjugacy classes of graded tensor functors  Vec$_G\to \C$.
\end{prop}
\begin{proof}
If $\C=\C_e\rtimes G$, then the category $\C_e$ is a $\C_e\rtimes G$-module category with action $[V,\sigma]\otimes W= V\otimes \sigma_*(W)$, see  \cite[Example 2.4]{tambara}.

Let $(\C_e,\odot,\mu)$ be an extension of $(\C_e,\otimes)$. By
Corollary \ref{mapa picard} $(\otimes, \C_\sigma)$ is the tensor
product of $\C_\sigma$ and $\C_e$ as $\C_e$-bimodule categories, and
by Proposition \ref{equivalencia sub} $\odot$ defines a
$\C_e$-module equivalence $\overline{\odot}$ such that the diagram
\[
\begin{diagram}
\node{\C_\sigma\times\C_e}\arrow{s,l}{\otimes} \arrow{e,t}{\odot}\node{\C_e}\\
\node{\C_\sigma}\arrow{ne,r}{\overline{\odot}}
\end{diagram}
\]commutes. Thus  $\overline{\odot}(V)= V\odot \unit$, for all $V\in \C_\sigma$. For each $\sigma\in G$, let $\unit_\sigma\in \C_\sigma$ be the unique object (up to isomorphism) such that $\unit_\sigma\odot \unit=\unit$.

Using the natural isomorphisms $\mu_{V_\sigma,X_e,\unit}$ and $\mu_{V_\sigma\otimes X_e, \unit_{\sigma^{-1}},\unit}$, we have

\begin{align*}
    V_\sigma\odot X_e &\cong (V_\sigma\otimes X_e)\odot \unit\\
                     &= (V_\sigma\otimes X_e)\odot (\unit_{\sigma^{-1}}\odot \unit)\\
                     &\cong (V_\sigma\otimes X_e\otimes\unit_{\sigma^{-1}})\odot \unit\\
                     &=V_\sigma\otimes X_e\otimes\unit_{\sigma^{-1}},
\end{align*}for all $V_\sigma \in\C_\sigma, X_e\in \C_e$. Then we can assume that $V_\sigma\odot X_e = V_\sigma\otimes X_e\otimes\unit_{\sigma^{-1}}$ for all $V_\sigma\in \C_\sigma, X_e\in \C_e, \sigma \in G$.

The natural isomorphisms \[\mu_{\unit_\sigma,\unit_h,\unit}: \unit_\sigma\otimes \unit_h\otimes \unit_{h^{-1}\sigma{-1}}  \to \unit_\sigma\otimes\unit_h\otimes \unit_{h^{-1}}\otimes \unit_{\sigma^{-1}}\] define isomorphisms $\gamma_{h^{-1},\sigma^{-1}}:\unit_{h^{-1}\sigma^{-1}}\to \unit_{h^{-1}}\otimes \unit_{\sigma^{-1}}$ for all $\sigma, h\in G$. Now, by the pentagonal equation \eqref{pentagono module cat}, the functor $F:$Vec$_G\to \C, k_\sigma\mapsto \unit_\sigma$ with the natural isomorphisms $\gamma_{\sigma,\tau}: \unit_{\sigma\tau}\to \unit_{\sigma}\otimes\unit_{\tau}$, defines a graded tensor functor. By Lemma \ref{lema semidir sii funtor graduado}, $\C$ is equivalent to a semi-direct product fusion category.

The construction of the graded tensor functor $(F,\gamma):$Vec$_G\to \C$, associated to a $\C$-extension $(\C_e,\odot)$ of $(\C_e,\otimes)$, shows that $(\C_e,\odot)$ is equivalent to $(\C_e,\odot^F)$.

Let $(F,\gamma), (F',\gamma'):$Vec$_G\to \C$ be graded functors, and
$(T,\eta): (\C_e,\odot^F)\to (\C_e,\odot^{F'})$ an equivalence of
$\C$-module categories. The functor $(T,\eta)$ is also an
equivalence of $\C_e$-module categories, so there is a
multiplicatively invertible object $U\in C_e$, such that $T(X_e)=
X_e\otimes U$ and $\eta_{V_e,X_e}=\id_{V_e\otimes X_e\otimes U}$ for
all $V_e, X_e\in \C_e$. The natural isomorphisms
\[\eta_{\unit_{\sigma},\unit, k_{\sigma^{-1}}}: F(k_\sigma)\otimes
F(k_{\sigma^{-1}})\otimes U \to F(k_\sigma)\otimes U\otimes
F'(k_{\sigma^{-1}}),\] define natural isomorphisms
\[\Theta_{\sigma^{-1}}: F(k_{\sigma^{-1}})\otimes U\to U\otimes
F'(k_{\sigma^{-1}}),\]for all $\sigma\in G$. Then the functor $F$
and $F'$ are conjugated by the pair $(U, \Theta_\sigma)_{\sigma \in
G}$,  where the commutativity of the diagram \eqref{pentagono
conjugados} follows from equation \eqref{penta funtor modulo} for
the functor module $(T,\eta)$.

\end{proof}

Let $\M$ be an indecomposable $\C_e$-module category and $\overline{\M}=\ind_{\C_e}^\C(\M)$. Then the fusion category $\C^*_{\overline{\M}}$ is $G^{op}$-graded.

If $A\in \C$ is an indecomposable semisimple algebra, then
${}_A\C_A$ is a fusion category and $B= {}^*A\otimes A\in {}_A\C_A$
is an algebra such that $_B(_A\C_A)_B=\C$, see \cite[Example
3.26]{finite-categories} (note that $_A\C_A$ is equivalent to
$\C_\M^*$ with reversed tensor product, where $\M =\C_A$). Using the
algebra $B$ we can describe a bijective correspondence between
module categories over $\C$ and module categories over $_A\C_A$, the
correspondence is given by
\begin{align*}
    Mod(\C) &\to Mod (_A\C_A)\\
    \M &\mapsto\- _A\M;\\
      Mod(_A\C_A) &\to Mod(\C)\\
      \N         &\mapsto\- _B\N
\end{align*}

Now we are ready to prove Theorem \ref{teor extensiones}.

\begin{proof}[Proof of Theorem \ref{teor extensiones}]
Let $A\in \C_e$ be an algebra such that $\M=(\C_e)_A$. Note that $_A\C_A$ is a $G$-graded fusion category. Suppose that $((\C_e)_A, \odot)$ is an extension of $((\C_e)_A,\otimes)$. Then $_A(\C_e)_A$ is an extension of $_A\C_A$, so by Proposition  \ref{lema extension facil}, $_A(\C)_A$ is a semidirect $G$-product fusion category.

Conversely, let $A\in \C_e$ such that $_A\C_A$ is a semidirect $G$-product fusion category. Then by Proposition \ref{lema extension facil} it defines an extension  $(_A(\C_e)_A,\odot)$ of $(_A(\C_e)_A,\otimes)$. For $B=\ ^*A\otimes A\in \C_e$ we have that $_B(_A(\C_e)_A)_B\cong \C$, so using this tensor equivalence we have a structure of $\C$-module category over $_B(_A(\C_e)_A)=(\C_e)_A$ which is an extension of $(\C_e)_A$.

The second part follows from the second part of Proposition
\ref{lema extension facil}.
\end{proof}

\begin{corol}\label{corol equiv}
Let $\C$ be a $G$-graded fusion category.  If $\M$ is an extension of an indecomposable $\C_e$-module category, then $\C_{\M}^*$ is a $G^{op}$-equivariantization of $(\C_e)_\M^*$. Conversely, if $\C$ is a $G^{op}$-equivariantization and $\M$ is an indecomposable $\C_e$-module category, then $\C_{\M}^*$ is a $G$-graded fusion category with $(\C_{\M}^*)_e\cong (\C_e)_\M^*$.
\end{corol}
\begin{proof}
Using the same argument as in proof of Theorem \ref{teor extensiones}, it is enough to see the case in which $\M$ is the $\C_e$-module category $(\C_e,\otimes)$.

By Theorem \ref{teor extensiones}, there exists an action
$*:\underline{G}\to \underline{\Aut}_\otimes (\C_e)$, such that
$\C=\C_e\rtimes G$. The category $\C$ is a $\C_e\rtimes G$-module
category with action $[V,\sigma]\otimes W= V\otimes \sigma_*(W)$,
see \cite[Proposition 3.2]{Non-grouptheo} or \cite[Example
2.4.]{tambara}. Moreover, the tensor category  $(\C_e\rtimes
G)_{\C_e}^*$, is monoidally equivalent to $(\C_e^G)^{\text{rev}}$
the $G$-equivariantization  of $\C$  with reversed tensor product,
see \cite[Proposition 3.2]{Non-grouptheo}. Conversely, $\C_e$ is a
$\C_e^G$-module category through the forgetful functor $\C_e^G\to
\C_e, (V,f)\mapsto V$, thus $(\C_e^G)^*_{\C_e}\cong (\C_e\rtimes
G)^{\text{rev}}$, see \cite[Proposition 3.2]{Non-grouptheo}.
\end{proof}

\begin{obs}\label{obs semi-final}

Let $H$ be a semisimple Hopf algebra such that the category of
$H$-comodules, Corep$(H)$, is a $G$-graded fusion category. Hence
the forgetful functor Corep$(H)\to \text{Vec}$ defines a structure
of Corep$(H)$-module category over Vec, and it is an extension as
Corep$(H)_e$-module category. By Corollary \ref{corol equiv}
Rep$(H)=\text{Corep}(H)^*_{\text{Vec}}$ is a
$G^{op}$-equivariantization. \smallbreak Conversely, if $\C$ is a
$G$-equivariantization of the category of representations of a
semisimple Hopf algebra $Q$, then $U \circ \text{Forg}$ is a fiber
functor of $\C$, where Forg$:\C\to \text{Rep}(Q), (V,f)\mapsto V$ is
the forgetful functor, and $U:\text{Rep}(Q)\to$ Vec is the forgetful
functor of Rep$(Q)$. Thus, by Tannaka-Krein reconstruction $\C$ is
the category of representations of a semisimple Hopf algebra $H$,
and by Corollary \ref{corol equiv} Rep$(H^*)=\C_{\text{Vec}}^*$ is a
$G^{op}$-graded fusion category.
\end{obs}

\begin{obs}\label{obs finales}
Theorem \ref{teor extensiones} and  Corollary \ref{corol main
result}, reduce the problem of constructing module categories over a
graded fusion category $\C=\bigoplus_{\sigma\in G}$, to the
following steps:
\begin{enumerate}
  \item Classifying the indecomposable $\C_e$-module categories.
  \item Finding the subgroup $S$ and the indecomposable $\C_e$-module categories $\N$, such that $\N$ is $S$-invariant.
  \item Determining if $\F_{\C_S}(\ind_{\C_e}^{\C_S}(\N),\ind_{\C_e}^{\C_S}(\N))$ is equivalent to a semi-direct $S^{op}$-product fusion category.
  \item Finding all graded functors from Vec$_{S^{op}}$ to $\F_{\C_S}(\N,\N)$, up to conjugation.
\end{enumerate}
\end{obs}

We shall briefly describe a way to solve the steps (3) and (4).

{\bf 1.}\ If $\C$ is a $G$-graded fusion category we have the following sequence of groups \begin{equation}\label{sucecion}
\begin{diagram}
  \node{U(\C_e)}\arrow{e,t}{\iota}\node{U(\C)}\arrow{e,t}{\text{\textbf{deg}}}\node{G}
\end{diagram}
\end{equation}
where $\iota$ is the inclusion and \textbf{deg}$(X)=\sigma$ if $X\in \C_\sigma$.

Note that a $G$-graded fusion category is a crossed product fusion
category if and only if the sequence \eqref{sucecion} is exact. By
Lemma \ref{lema semidir sii funtor graduado}, $\C$ is monoidally
equivalent to a semi-direct product fusion category if and only if
the sequence \eqref{sucecion} is exact and there exists a
\emph{splitting morphism} of \textbf{deg}, \textit{i.e.}, there
exists a group morphism $\pi: G\to U(\C)$ with
\textbf{deg}$\circ\pi=\id_G$  (thus \eqref{sucecion} is a split
extension),  such that the obstruction $\omega(\D)$ of the fusion
subcategory  $\D\subset \C$ generated by $\{\pi(\sigma)\}_{\sigma\in
G}$ is zero (see Subsection \ref{pointed y obstruccion} for the
definition of the obstruction of a pointed fusion category).

{\bf 2.} Let $G$ and $N$ be groups and $*:G\to \Aut(N)$ a group
morphism. A function  $\theta: G\to N, \sigma,\mapsto \theta_\sigma$
is called a 1-cocycle from $G$ to $N$ if
\[\theta_\sigma\theta_\tau=\theta_\sigma\sigma_*( \theta_{\tau})\]
for all $\sigma,\tau \in G$. The set of all 1-cocycles is denoted by
$Z^1(G,N)$. The group $N$ acts  over $Z^1(G,N)$ by \[(u\cdot
\theta)_\sigma=u\theta_\sigma \sigma_*(u^{-1}),\] for all $u\in N,
\sigma \in G, \theta \in Z^1(G,N)$. The set of orbits $Z^{1}(G,N)/G$
is denoted by $H^{1}(G,N)$, and two 1-cocycles  in the same orbit
are called cohomologous. \smallbreak Suppose that $\C=\C_e\rtimes G$
is a semi-direct product fusion category. Hence
$U(\C)=U(\C_e)\rtimes G$, and the sequence \eqref{sucecion} is a
semidirect product extension. Recall that there is a bijective
correspondence between 1-cocycles from $G$ to $U(\C_e)$ and
splitting homomorphisms of \textbf{deg}. In fact, given  $\theta\in
Z^1(G,U(\C_e))$ the map  $\pi_{\theta}: G\to U(\C_e)\rtimes G,
\sigma \mapsto [\theta_\sigma,\sigma]$  is a splitting morphism of
\textbf{deg}. Conversely, if $\pi_{\theta}: G\to U(\C_e)\rtimes G,
\sigma \mapsto [\theta_\sigma,\sigma]$ is a splitting morphism, the
map $\theta: G\to U(\C_e), \sigma \mapsto \theta_\sigma$ is a
1-cocycle. Moreover, the splitting morphisms are conjugate by an
element in $U(\C_e)$ if and only if the 1-cocycles associated are
cohomologous.

\begin{defin}
Let $\omega\in H^3(U(\C),k^*)$ be the obstruction of the maximal pointed fusion subcategory of $\C$, and let $\theta: G\to U(\C_e)$ be a 1-cocycle. The obstruction of  $\theta$ is defined as the cohomology class of $\omega|_{X_\theta}\in H^3(X_\theta,k^*)$, where $X_\theta =\{[\theta_\sigma,\sigma]\}_{\sigma\in G}\subset U(\C_e)\rtimes G$.
\end{defin}

Without loss of generality we may assume that $\C$ is skeletal. Thus
there is a unique 3-cocycle $\omega \in Z^3(U(\C),k^*)$, such that
$\alpha_{\sigma,\tau, \rho}=\omega(\sigma,\tau
\rho)\id_{\sigma\otimes \tau \otimes \rho}$, and the cohomology
class of $\omega$ is the obstruction of the maximal pointed fusion
subcategory of $\C$.

Let $\theta \in Z^1(G,U(\C))$ such that $[\omega|_{X_\theta}]=0$, and  $L_\omega^\theta =\{\gamma \in C^2(G,k^*)| \delta(\gamma) =\omega|_{X_\theta}\}$. Then  a graded tensor functor  $(F,\gamma):$Vec$_G\to \C$ with $k_\sigma\mapsto [\theta_\sigma,\sigma]$ defines an element $\gamma' \in L_\omega^\theta $ by \begin{equation}\label{formula dos}\gamma_{k_\sigma, k_\tau}=\gamma'(\sigma,\tau)\id_{F(k_\sigma)\otimes F(k_\tau)},\end{equation}for all $\sigma,\tau \in G$. Conversely, using the formula \eqref{formula dos} every element in $L_\omega^\theta$ defines a graded functor with $F(k_\sigma)=[\theta_\sigma,\sigma]$ for all $\sigma \in G$.

The next proposition is a consequence of the previous discussion and Proposition \ref{lema extension facil}.

\begin{prop}\label{propo decripcion extension en group-theo data}
Let $*:\underline{G}\to \underline{\Aut}_\otimes(\C_e)$ be a group
action over $\C_e$. Then there is a bijective correspondence between
$\C_e\rtimes G$-extension of  $(\C_e,\otimes)$ and pairs $(\theta,
\gamma)$, where $\theta \in Z^1(G,U(\C_e))$ is a 1-cocycle with
obstruction zero and $\gamma \in L_\omega^\theta =\{\gamma \in
C^2(G,k^*)| \delta(\gamma) =\omega|_{X_\theta}\}$. Two pairs
$(\theta,\gamma^{\theta}), (\nu,\gamma^{\nu})$ define equivalent
$\C$-extensions if and only if there is $u\in U(\C_e)$ and
$\kappa\in C^1(G,k^*)$, such that $\theta=u\cdot \nu $, and
\[
\gamma^\nu(\sigma,\tau)=\gamma^{\theta}(\sigma,\tau)\delta(\kappa)(\sigma,\tau) \omega([\theta_\sigma,\sigma],[\theta_\tau,\tau],[u,e]) \omega([\theta_\sigma,\sigma],[u,e],[\nu_\tau,\tau])
\]for all $\sigma,\tau \in G$.
\end{prop}
\qed

\begin{obs}
The group $Z^2(G,k^*)$ acts on $L_\omega^\theta$ by multiplication, and this action is free and transitive. Thus $L_\omega^\theta$ is a torsor over $Z^2(G,k^*)$, and there is  a (non natural) bijective correspondence between $L_\omega^\theta$ and $Z^2(G,k^*)$. Also, note that $\gamma, \gamma' \in L_\omega^\theta$ define tensor equivalent graded functors if and only if $\gamma\gamma^{-1}\in Z^2(G,k^*)$. Then set of tensor equivalence classes of graded functor with splitting homomorphism defined by $\theta\in Z^1(G,U(\C_e))$ is a torsor over $H^2(G,k^*)$.
\end{obs}
\bibliographystyle{amsalpha}

\begin{thebibliography}{A}
\bibitem{BaKi} {\sc B. Bakalov} and {\sc A. Kirillov Jr.}, \emph{Lectures on Tensor categories and
    modular functors}, AMS, (2001).

\bibitem{D} {\sc P. Deligne}, \emph{Cat\'egories tannakiennes}, in : \emph{The Grothendieck
    Festschrift}, Vol. II, Progr. Math., \textbf{87}, Birkhäuser, Boston, MA, 1990, 111-195.

\bibitem{DGNO} \textsc{V. Drinfeld}, \textsc{S. Gelaki}, \textsc{D. Nikshych}, and \textsc{V. Ostrik},
    \emph{On Braided Fusion Categories I}, Selecta Math. N.S. 16, 1 (2010) 1--119.
\bibitem{ENO} {\sc P. Etingof}, {\sc D. Nikshych} and {\sc V. Ostrik}, \emph{On fusion categories}, Ann. Math., 162 (2005), 581--642.

\bibitem{ENO2} {\sc P. Etingof}, {\sc D. Nikshych} and {\sc V. Ostrik}, \emph{Weakly group-theoretical and solvable fusion
categories}, Adv. Math., 226 (2011) 176--205.

\bibitem{ENO3} {\sc P. Etingof}, {\sc D. Nikshych} and {\sc V. Ostrik}, \emph{Fusion categories and homotopy theory}, Quantum Topol., 1 (3)
(2010) 209--273.

\bibitem{finite-categories} {\sc P. Etingof} and {\sc V. Ostrik},
\emph{Finite tensor categories}, Mosc. Math. J., 4 (2004), 627--654,
782--783.

\bibitem{G} {\sc C. Galindo},
\emph{Clifford theory for tensor categories}, J. Lond. Math. Soc.,
(2) 83 (2011) 57--78.

\bibitem{CG2} {\sc C. Galindo}, \emph{Crossed product tensor categories}, J. Algebra, 337 (2011)
233--252.

\bibitem{GNN} {\sc  S. Gelaki}, {\sc D. Naidu} and {\sc D. Nikshych},
\emph{Centers of graded fusion categories}, g, Alg. Number Th., 3
(2009), no. 8, 959-990


\bibitem{Justin} \textsc{J. Greenough}, \emph{Monoidal 2-structure of Bimodule Categories},  J. Algebra, 324 (2010), 1818-1859.

\bibitem{2Gen-TY} {\sc D. Jordan} and {\sc E. Larson} , \emph{On the classification of certain fusion categories}, Journal of Noncommutative Geometry, Vol 3, Issue 3, 2009, 481-499.

\bibitem{Gen-TY} {\sc J. Liptrap}, \emph{Generalized Tambara-Yamagami categories}, preprint arXiv:1002.3166v2.

\bibitem{MM} {\sc E. Meir},  {\sc E. Musicantov},
\emph{Module categories over graded fusion categories}, preprint \texttt{arxiv:1010.4333}.


\bibitem{Naidu}  {\sc D. Naidu},
\emph{Categorical Morita equivalence for group-theoretical
categories.} Commun. Alg.,  35 (2007), 3544-3565



\bibitem{Non-grouptheo}  {\sc D. Nikshych},
\emph{Non group-theoretical semisimple Hopf algebras from group
actions on fusion categories.} Selecta Math. 14 (2008), 145-161.

\bibitem{O1} {\sc V. Ostrik},
\emph{Module categories, weak Hopf algebras and modular invariants},
Transform. Groups, 8 (2003), 177-206.

\bibitem{O2} {\sc V. Ostrik}, \emph{Module categories over the Drinfeld double of a
    finite group}, Int.\ Math.\ Res.\ Not. (2003) no.\ 27, 1507-1520.

\bibitem{action scha} {\sc P. Schauenburg}, \emph{Actions of monoidal categories, and generalized Hopf smash products}, J. Algebra 270 (2003), 521--563.


\bibitem{TY}  {\sc D. Tambara} y {\sc S. Yamagami},
\emph{Tensor categories with fusion rules of self-duality for finite
abelian groups}, J. Algebra, 209 (1998), 692--707.

\bibitem{tambara}  {\sc D. Tambara}, \emph{Invariants and semi-direct products for finite group
    actions on tensor categories}, J. Math. Soc. Japan, 53 (2001), 429--456.

\bibitem{Ulbrich}  {\sc Ulbrich, K-H}, \emph{ Fiber functor of finite dimensional comodules},
    Manuscripta Math., 65 (1989), 39--46.


\bibitem{watts}  {\sc  C. E. Watts},  \emph{ Intrinsic characterizations of some additive funtors}, Proc. Amer. Math. Soc.
\textbf{11}  (1960), 5--8.




\end{thebibliography}

\end{document}